\documentclass[a4page,twoside]{amsart}
\usepackage{lmodern}

\usepackage[dvipsnames,svgnames,x11names,hyperref]{xcolor}

\usepackage{amsmath,amssymb,mathtools}
\usepackage{enumitem}
\setitemize{label=$-$}
\setenumerate{label=(\roman*)}

\usepackage[margin=3.6cm]{geometry}
\usepackage{tikz-cd}
\usetikzlibrary{matrix,calc,positioning,arrows,decorations.pathreplacing,patterns,arrows}

\newtheorem{thm}{Theorem}[section]
\newtheorem*{thm*}{Theorem}
\newtheorem{lem}[thm]{Lemma}
\newtheorem{prop}[thm]{Proposition}
\newtheorem{cor}[thm]{Corollary}
\newtheorem*{cor*}{Corollary}

\theoremstyle{definition}
\newtheorem{defn}[thm]{Definition}
\newtheorem{cons}[thm]{Construction}
\newtheorem{assum}[thm]{Assumption}
\theoremstyle{remark}
\newtheorem{example}[thm]{Example}
\newtheorem{rem}[thm]{Remark}

\numberwithin{equation}{section}

\DeclareRobustCommand\sm{\mathbin{\mathpalette\smaux\relax}}
\newcommand\smaux[2]{\mspace{-6mu}
\raisebox{\rsmraise{#1}\depth}{\rotatebox[origin=c]{-25}{$#1\smallsetminus$}}
 \mspace{-6mu}
}
\newcommand\rsmraise[1]{%
  \ifx#1\displaystyle .4\else
    \ifx#1\textstyle .4\else
      \ifx#1\scriptstyle .3\else
        .45%
      \fi
    \fi
  \fi}

\iffalse

\usepackage[backend=biber,
style=alphabetic,
sorting=nyt,
giveninits=true,
maxnames=10,
backref,
backrefstyle=none,
]{biblatex}

\renewbibmacro{in:}{%
  \ifentrytype{article}{}{\printtext{\bibstring{in}\intitlepunct}}}

\DeclareFieldFormat[article,incollection,inproceedings]{title}{#1\isdot}
\DeclareFieldFormat[unpublished]{title}{\textit{#1}\isdot}

\DefineBibliographyStrings{english}{%
    backrefpage  = {p.}, 
    backrefpages = {pp.} 
}

\fi

\usepackage[colorlinks,citecolor=Mahogany,linkcolor=Mahogany,urlcolor=Mahogany,filecolor=Mahogany]{hyperref}
\usepackage[capitalize]{cleveref}

\newcommand{\wt}[1]{{\widetilde{#1}}}

\newcommand{\op}{^\mathrm{op}}
\newcommand{\cat}[1]{\mathrm{#1}}
\newcommand{\on}{\operatorname}

\newcommand{\Map}{\mathrm{Map}}

\newcommand{\Hom}{\mathrm{Hom}}
\newcommand{\id}{\mathrm{id}}
\newcommand{\Conf}{\mathrm{Conf}}
\newcommand{\con}{\mathrm{con}}

\newcommand{\Z}{\mathbf{Z}}
\newcommand{\Q}{\mathbf{Q}}
\newcommand{\R}{\mathbf{R}}
\newcommand{\K}{\mathbf{K}}
\renewcommand{\k}{\mathbf{k}}
\newcommand{\C}{\mathbf{C}}

\newcommand{\G}{\mathcal{G}}

\newcommand{\ra}{\longrightarrow}
\newcommand{\hra}{\hookrightarrow}
\newcommand{\sra}{\twoheadrightarrow}

\newcommand{\fin}{\mathrm{Fin}}
\newcommand{\sfin}{\Delta\mathrm{Fin}}
\newcommand{\AS}{\mathrm{AS}}
\newcommand{\FM}{\mathrm{FM}}

\newcommand{\blr}{\mathrm{Bl}^{\R}}
\newcommand{\blc}{\mathrm{Bl}^{\C}}
\newcommand{\bll}{\mathrm{Bl}^{log}} 

\newcommand{\bl}{\mathrm{Bl}} 
\newcommand{\kn}{\mathrm{KN}}

\renewcommand{\for}{\mathrm{For}} 
\newcommand{\forset}{\mathsf{For}} 
\newcommand{\logX}{\mathsf{X}}
\newcommand{\logY}{\mathsf{Y}}
\newcommand{\logZ}{\mathsf{Z}}
\newcommand{\logFM}{\mathsf{FM}}

\newcommand{\comment}[1]{{\color{blue}{#1}}}

\title{An algebro-geometric model for the configuration category}

\author{Pedro Boavida de Brito}
\address{Dept. of Mathematics, Instituto Superior Tecnico, Univ.~of Lisbon, Av.~Rovisco Pais, Lisboa, Portugal}%
\email{pedrobbrito@tecnico.ulisboa.pt}

\author{Geoffroy Horel}
\address{Université Sorbonne Paris Nord, Laboratoire Analyse, Géométrie et Applications, CNRS (UMR 7539), 93430, Villetaneuse, France.}
\email{horel@math.univ-paris13.fr}

\author[D. Kosanovi\'{c}]{Danica Kosanovi\'{c}}
\address{Department of Mathematics, ETH Zürich, Rämistrasse 101, 8092 Zürich, Switzerland}
\email{danica.kosanovic@math.ethz.ch}

\date{\today}

\begin{document}

\begin{abstract}
    Using log-geometry, we construct a model for the configuration category of a smooth algebraic variety. As an application, we prove the formality of certain configuration spaces.

\end{abstract}

\keywords{}
\maketitle

\section{Introduction}

Given a finite set $S$ and a manifold $M$, we denote by $\Conf_S(M)$ the space of embeddings of $S$ in $M$. These spaces play a fundamental role in differential and geometric topology. They have also been of interest in arithmetic geometry.  In particular, there is a tradition of studying the motivic structure on the fundamental group of configuration spaces of algebraic curves (see for example \cite{iharabraid,kohnooda}).

From the point of view of topology, an important fact is that the collection of these configuration spaces for all finite sets $S$ are related by a rich structure. The most obvious part of this structure is given by the \emph{forgetful} maps
\[\Conf_S(M)\to\Conf_{S-s}(M)\]
that discard one of the points $s \in S$. In the other direction, there are \emph{insertion} maps
\[\Conf_S(M)\times\Conf_T(D)\to\Conf_{(S-s)\sqcup T}(M).\]
that replace the point labelled $s$ by a configuration of $T$ points in a small enough ball $D$ around that point. This notation is somewhat abusive as the source is usually not a product but rather the total space of a fiber bundle over $\Conf_S(M)$ whose fiber is $\Conf_T(D)$. There are also some choices involved in making this map (but these choices are irrelevant in the homotopy category). When $M$ is a smooth manifold, there is a beautiful geometric incarnation of these insertion maps: replacing configuration spaces by their Axelrod-Singer compactifications \cite{axelrodsinger}, the insertion map becomes simply the inclusion of a stratum. This point of view has one further merit which interests us in this paper. The Axelrod-Singer compactification has an algebraic predecessor, due to Fulton-MacPherson \cite{fultonmacpherson}, and so there is a natural algebro-geometric analog of the insertion maps when $M$ is a smooth algebraic variety.


When the manifold is $\R^n$, the forgetful and insertion operations are the structure maps of the little $n$-disks operad. For a more general manifold, a convenient object that packages all of this structure is the ``configuration category'' of the manifold \cite{boavidaspaces}. As the name suggests, the configuration category is a category or, more accurately, an $(\infty,1)$-category. Roughly speaking, its objects are configuration of any number of points in the manifold and its morphisms are paths of configurations in which points are allowed to collide, but if that happens, they have to stay together until the end. As it turns out, the little $n$-disks operad and the configuration category of $\R^n$ are equivalent pieces of data in a homotopical sense \cite{boavidaspaces}. When $M$ is a differentiable manifold, the configuration category has an ``Axelrod-Singer model'' where the insertion and forgetful maps correspond to source maps from morphism to object spaces. 


It has long been conjectured that the little disks operad has an algebro-geometric origin. That goes back at least to Beilinson, who suggested it in a letter to Kontsevich, and Morava who asks \cite{Morava}: is the little disks operad defined over $\Q$? This conjecture has been settled in the work of Vaintrob \cite{vaintrobmoduli} and Dupont-Panzer-Pym \cite{dupontlogarithmic} using the theory of log-schemes. More precisely, these authors construct an operad in a category of log-schemes over the complex numbers whose Kato-Nakayama realization is a model for the (framed) little $2$-disks operad.

\medskip
The goal of the present paper is to construct an algebro-geometric model for the configuration category of \emph{any} algebraic variety. Again it is built using the theory of log-schemes. To state our main theorem, we regard the configuration category as a space-valued functor on $\Delta\cat{Fin}$, the category of simplices of the nerve of the category of finite sets (remark \ref{rem:con}). Our main theorem is the following.

\begin{thm*}
Let $X$ be a smooth algebraic variety over a field $K$. 
There is a functor from $\Delta\cat{Fin}\op$ to log-schemes over $K$ -- denoted $\con^{log}(X)$ -- with the following properties.
\begin{enumerate}
\item Its value on a finite set $S$ is the Fulton-MacPherson scheme $\FM_S(X)$ with its divisorial log-structure.
\item If the base field is $\C$ then the Kato-Nakayama realization of $\con^{log}(X)$ is isomorphic to the Axelrod-Singer model $\con^{AS}(X(\C))$ for the configuration category of $X(\C)$.
\end{enumerate}
\end{thm*}

As a consequence, we obtain a log-model for the little $n$-disks operad for every $n$, suitably interpreted. By that we mean, we regard the little disks operad as a functor on a category of forests, in contrast to Dupont-Panzer-Pym for $n = 2$, who take a strict view of the operad but expand the classical meaning of \emph{morphism} of log scheme.

\medskip
An immediate corollary of the theorem is the following result which was the original motivation behind this paper.
\begin{cor*}
Let $X$ be a smooth scheme over $\Q$. The profinite completion of $\con^{AS}(X(\C))$, the Axelrod-Singer model for the configuration category of $X(\C)$, has an action of $\mathrm{Gal}(\overline{\Q}/\Q)$. This action extends the natural Galois action on the profinite completion of the configuration spaces of points in $X(\C)$.
\end{cor*}
Note that when $X$ is affine $n$-space, we obtain a homotopy coherent Galois action on the profinite completion of $E_{2n}$ as opposed to an action which is coherent up to homotopy. (The latter, in the case $n = 1$, amounts to a map from $\mathrm{Gal}(\overline{\Q}/\Q)$ to the group of homotopy automorphisms of the profinite completion of $E_2$ which is isomorphic to the Grothendieck-Teichmuller group by \cite{horel}.)

Combining these Galois actions with the additivity theorem for configuration categories we can deduce formality, in the sense of rational homotopy theory, for configuration spaces in $X(\C) \times \R$ where $X$ is a smooth scheme over $\C$ (Corollary \ref{cor}). In future work, we will use this to prove that the homological Goodwillie-Weiss spectral sequence for knots in $X(\C) \times \R$ collapses rationally.

\subsection*{Acknowledgements}
We thank Dmitry Vaintrob, Dan Petersen, Clément Dupont for useful conversations. P.B. was supported by FCT 2021.01497.CEECIND. G.H. was supported by the project  ANR-20-CE40-0016 HighAGT. Both P.B. and G.H. are supported by a project PHC Pessoa.

\section{Categories of Forests}

\subsection{Forests}
\begin{defn}
	Fix a finite set $S$. A \emph{forest} on $S$ is a set $\Phi$ of non-empty subsets of $S$ such that
\begin{enumerate}
	\item Every one-point subset of $S$ is contained in $\Phi$, and is called a \emph{leaf} of $\Phi$.
	\item Any two elements of $\Phi$ are either disjoint or comparable for the inclusion relation.
\end{enumerate}  
	A \emph{tree} on a set $U$ is a forest containing the set $U$, which is called the \emph{root} of $\Phi$. Moreover, for a forest $\Phi$ on $S$ and a maximal element $U$ of $\Phi$, the set of elements in $\Phi$ contained in $U$ forms a tree on $U$. Clearly, any forest is a disjoint union of trees, called the \emph{trees of the forest} $\Phi$. 
\end{defn}

\begin{rem}
	Given a forest $\Phi$ on $S$ as above, we associate to it a graph $G_\Phi$ 
as follows.
\begin{enumerate}
	\item If $\Phi$ consists of trees $\Phi_1,\dots,\Phi_n$, then $G_\Phi$ is the disjoint union $G_\Phi=\bigsqcup_{i=1}^nG_{\Phi_i}$.
	\item If $\Phi$ is the unique tree on a one-element set $S=\{s\}$, then $G_\Phi$ has two vertices connected by one edge. One vertex is a leaf decorated by $s$ and the other is a non-decorated root.
	\item If $\Phi$ is a tree on a set $U$ with strictly more than one element, then we construct a contractible graph (a tree in graphical sense) whose vertices are the elements of $\Phi$. We add one edge connecting vertices $A$ and $B$ if and only if they are comparable with respect to inclusion and there are no other elements of $\Phi$ between them (i.e.\ if say $A\subsetneq B$, then there is no element $C$ in $\Phi$ such that $A\subsetneq C\subsetneq B$). Vertices with one incident edge correspond either to $U$ (the root), or to the one-element sets (the leaves). There are no \emph{unary} vertices, i.e. vertices having two incident edges.
\end{enumerate}
\end{rem}

For a finite set $S$, we write $P(S)$ for the poset of non-empty subsets of $S$. The forests on a set $S$ form a poset $\forset(S)$ under the inclusion relation, as subsets of $P(S)$. This poset has a minimal element given by the forest with only one point sets.

Observe that if $j\colon S\hra T$ is an injection, then any forest on $T$ induces a forest on $S$ by intersecting it with $S$. This operation is clearly a map of posets 
\[j^{-1}\colon\forset(T)\to\forset(S),\]
so that $S\mapsto \forset(S)$ is a contravariant functor from finite sets and injections to posets. 

We define $\for'$ to be the Grothendieck construction of this functor. That is, an object of $\for'$ is a pair $(S,\Phi)$ with $S$ a finite set and $\Phi$ a forest on $S$. A morphism from $(S,\Phi)$ to $(T,\Psi)$ is an injection $j\colon S\to T$ such that $\Phi \subset j^{-1}\Psi$. The actual category $\for$ of forests that we will use is a certain quotient of $\for'$, defined as follows. 

\begin{defn}\label{defi : category of forests}
	The category $\for$ has the same objects as $\for'$, and if $(S,\Phi)$ and $(T,\Psi)$ are objects of $\for'$, we define
\[\Hom_{\for}((S,\Phi),(T,\Psi))\coloneqq\Hom_{\for'}((S,\Phi),(T,\Psi))/\sim,\]
	where the equivalence relation $\sim$ identifies two injective maps $i,j\colon S\to T$ if $i^{-1}\Psi=j^{-1}\Psi$. The composition is induced from the composition in $\for'$.
\end{defn}

Concretely, we think of an object in $\for$ as a contractible graph whose vertices are not unary. (An unary vertex is one having exactly two adjacent edges.) A morphism in the opposite category corresponds to contracting inner edges, making sure the resulting forest has no unary vertices, or to removing a root edge. This is reminiscent of \cite{dend}. We can make this more precise as follows. 

\begin{lem}\label{lem: forests}
The category $\for$ is equivalent to the category whose objects are posets $V$ satisfying 
\begin{enumerate}
\item Given $v \in V$, the set $\{w | w \leq v\}$ is totally ordered.
\item For $S$ the set of maximal elements of $V$, the map $V \to P(S)$ which to a vertex $v$ assigns the set of maximal elements that are bigger than $v$, is injective.
\end{enumerate}
and whose morphisms are injective maps of posets preserving \emph{independence}, i.e., sending uncomparable elements to uncomparable elements.
\end{lem}
\begin{proof}
Denote by $C$ the category in the statement. Given a forest $(S, \Phi)$, the set $\Phi$ (whose elements we call vertices) has a poset structure given by reverse inclusion. Then $\Phi$ clearly satisfies condition (i) and condition (ii) rephrases the condition that the forest has no unary vertices.

Conversely, from any finite poset $V$ satisfying conditions (i) and (ii) we may define a forest as the image of the map $V \to P(S)$ in (ii), where $S$ is the set of maximal elements of $V$. This establishes a bijection between forests in the sense above and posets satisfying conditions (i) and (ii). (Note that for such posets, the map in (ii) identifies the order on $\Phi$ with the opposite order on $P(S)$.)

The assignment $(S, \Phi) \mapsto \Phi$ defines a functor $\for^\prime \to C$: a morphism $j : (S, \Phi) \to (T, \Psi)$ gives rise to an injective map of posets $\Phi \to j^{-1}(\Psi) \to \Psi$ which preserves independence. It is clear that this functor factors through $\for$. Conversely, we describe the functor $C \to \for$ as follows. Given a morphism $f : V \to W$ in $C$, we get a corresponding map from the set $S$ of maximal elements of $V$ to the set $R$ of maximal elements of $W$ sending $s \in S$ to \emph{any} $r \in R$ such that $r \geq f(s)$. The resulting map $S \to R$ is a morphism in $\for$ (but not in $\for^\prime$). The two functors that we have constructed give an inverse equivalence between $C$ and $\for$.
\end{proof}

\begin{rem}\label{rem : forests vs nests}
	The set of forests is in bijection with the set of nests in \cite{axelrodsinger} and \cite[Section 3]{boavidaspaces}, in the ``screen compactification model'' built out of the spaces  $\AS_n(X)$. Indeed, given a nest we may add to it all the one-point subsets and obtain a forest. Conversely, given a forest, we can remove all the one-point subsets and obtain a nest.
\end{rem}

\subsection{Forests with levels}\label{sec:forestswlevels}

Consider the category $\sfin$ of simplices of the nerve of the category of finite sets. Concretely, an object of $\sfin$ is a functor $\alpha\colon[k]\to\fin$ for some $k$, i.e.\ a sequence
\[ S_0\xrightarrow{f_0}S_1\xrightarrow{f_1}\ldots\xrightarrow{f_{k-1}}S_k\]
of composable maps of finite sets. A morphism from $\alpha\colon [k]\to\fin$ to $\beta\colon[l]\to\fin$ is a map $\delta\colon[k]\to[l]$ in $\Delta$ such that $\alpha= \beta\circ\delta$. Concretely, a morphism corresponding to an elementary degeneracy insert an identity, and a morphism corresponding to an elementary face is either a composition of two maps or it forgets the first or last map.

We use the characterization of $\for$ in lemma \ref{lem: forests} to describe a functor $$F : \Delta\fin \to \for \; .$$ An object $\alpha$ of $\Delta\fin$ determines an object $F(\alpha)$ of $\for$, as follows. As a set $F(\alpha)$ is the quotient of the disjoint union of the $S_i$ by the relation $x \sim f_i(x)$ if the preimage of $x$ under $f_i$ is $\{x\}$. As for the order relation on $F(\alpha)$, we take the order induced by declaring $x \leq y$ in the disjoint union of the $S_i$ if $x = g(y)$ where $g$ is some composition of maps $f_j$. 

For a morphism $\alpha \to \alpha^\prime$ in $\Delta\fin$ corresponding to a surjection $\delta : [k] \to [\ell]$, we define $F(\alpha) \to F(\alpha^\prime)$ to be the identity. For a morphism $\alpha \to \alpha^\prime$ corresponding to an injection $\delta : [k] \to [\ell]$, the induced map
\[
\coprod_{i \in [k]} S_i \to \coprod_{i \in [\ell]} S_i
\]
respects the order relation and passes to the quotient, defining a map $F(\alpha) \to F(\alpha^\prime)$. It is clear that $F$ respects compositions.

\section{Blow-ups}

\subsection{Complex and oriented real blow-ups}

We follow quite closely the treatment of blow-ups of \cite[Section 8.3]{bergstromhyperelliptic}. Let $\mathbf{P}^{n-1}_\C=(\C^n\sm0)/\C^{\times}$ and $S^{n-1}\cong(\R^n\sm0)/\R_{>0}$. 

\begin{defn}
	For a topological space $X$ and a continuous map $\sigma\colon X\to \C^n$, we define the \emph{complex blow-up} along $\sigma$ as the space
\[
	\blc_\sigma(X)\coloneqq\{(x,[w])\in X\times\mathbf{P}^{n-1}_\C : (\exists\alpha\in\C)\; \sigma(x)=\alpha w\}.
\]
\end{defn}

\begin{defn}
	For a topological space $X$ and a continuous map $\sigma\colon X\to \R^n$, we define the \emph{real oriented blow-up} along $\sigma$ as the space
\[
	\blr_\sigma(X)\coloneqq\{(x,[w])\in X\times S^{n-1}: (\exists\alpha\in\R_{\geq 0})\; \sigma(x)=\alpha w\}.
\]
\end{defn}

In either case there is a \emph{blow-down map} $bl_\sigma$ to $X$ given by $bl_\sigma(x,[w])=x$. This is an isomorphism away from the zero-set of $\sigma$, whereas over any zero of $\sigma$ the fiber is the projective space $\mathbf{P}^{n-1}_\C$, or the sphere $S^{n-1}$ respectively. The zero-set of $\sigma$ is called the \emph{blow-up locus}, and its inverse image under $bl_\sigma$ is the \emph{exceptional divisor}. 
  
If $Y\subseteq X$ we define the \emph{dominant transform} $\wt{Y}\subseteq\blc_\sigma(X)$ as the closure of $$bl_\sigma^{-1}(\{y\in Y: \sigma(y)\neq0\})$$ in case this is nonempty, i.e. $Y$ is not contained in $\sigma^{-1}(0)$; otherwise, define $\wt{Y}$ as $bl_\sigma^{-1}(Y)$. 

\begin{example}
	For $X=\C^n$ and $\sigma=Id\colon\C^n\to\C^n$ we have $\blc_{\sigma}(\C^n)=\blc_0(\C^n)$, the (complex projective) blow-up of $\C^n$ at the origin. More generally, for $X$ a complex algebraic variety and $\sigma:X\times\C^n\to\C$ given by $\sigma(x,u)=u$, the blow-up $\blc_{\sigma}( X\times\C^n)$ can be identified with
	\[
		\blc_{X\times0}(X\times\C^n)\cong X\times\blc_0(\C^n).
	\]
\end{example}
\begin{example}\label{exam:AK}
	For $X=\R^{m+n}$ and $\sigma=pr\colon\R^{m+n}\to\R^n$ the projection to the last $n$ coordinates, we have the zero-locus $\R^m\times0$. Then there is a map
	\[
		\blr_{\sigma}(\R^{m+n})
  =\{((x,y),[w])\in X\times S^{n-1}: (\exists\alpha\in\R_{>0})\; y=\alpha w\}
  \ra
  \R^m\times\R_{\geq0}\times S^{n-1}
	\]
	defined by $((x,y),[w])\mapsto (x,|y|,[w])$. This is clearly continuous and surjective, and it is injective as well: if $|y_1|=|y_2|$ and both $y_1$ and $y_2$ are positively proportional to a fixed $w$, then $y_1=y_2$. The blow-down map corresponds to $(x,r,[w])\mapsto (x,\frac{r}{|w|}w)$.
\end{example}

We can make similar definitions when $\sigma$ is, more generally, a section of a complex or real vector bundle $E$: we use the formulas above in local charts in which $E$ is trivial and glue those spaces together. In each case there is a canonical blow-down map to $X$, whose fiber over a non-zero point of $\sigma$ is homeomorphic to a point, whereas the fiber over a zero of $\sigma$ is homeomorphic to a copy of $\mathbf{P}_\C^{n-1}$, respectively $S^{n-1}$. Then the following is immediate.

\begin{prop}
	If $E\to X$ is a complex vector bundle and $0_E\colon X\to E$ is the zero section, then $\blc_{0_E}(X)$ is the projectivization of $E$. Similarly, if $E\to X$ is a real vector bundle, then $\blr_{0_E}(X)$ is the sphere bundle of $E$.
\end{prop}

\begin{rem}
	For $\sigma\colon X\to\C^n$ that is not identically zero, we can view $\blc_\sigma(X)$ as the {dominant transform} of the graph of $\sigma$ under the blow-up of $X\times0\subset X\times\C^n$.
		Similarly, for a section $\sigma\colon X\to E$ of a bundle, the space $\blc_\sigma(X)$ is the dominant transform of the graph of $\sigma$ under the complex projective blow-up of $E$ along the zero section. And similarly in the real oriented case.
\end{rem}

\subsection{Iterated  real blow-ups}
The blow-up construction is functorial in the following way.

\begin{prop}\label{prop : cartesian square}
    Given a map $f\colon X\to Y$ of topological spaces and a section $\sigma\colon X\to E$ of a real vector bundle, consider the pullback bundle $f^*E$ and the section $f^*\sigma$, and form $\blr_{f^*\sigma}(Y)$. Then the following is a cartesian square:
\[
\begin{tikzcd}
	\blr_{f^*\sigma}(Y)\rar\dar & \blr_{\sigma}(X)\dar\\
	Y\rar{f} &X.
\end{tikzcd}
\]
\end{prop}

\begin{proof}
    This appears in \cite[Paragraph 8.3.5]{bergstromhyperelliptic}. It can be checked locally on $X$ so we may assume that $E$ is a trivial bundle in which case this is a trivial computation.
\end{proof}

Now let $X$ be a topological space, $\{E_i\to X\}_{1\leq i\leq n}$ a collection of real vector bundles and $\{\sigma_i\colon X\to E_i\}_{1\leq i\leq n}$ a collection of sections. We abusively denote by
\begin{equation}\label{eq:iterated-blow-up}
	\blr_{\sigma_n}\ldots\blr_{\sigma_1}(X)
\end{equation}
the iterated blow-up construction where at each stage we pullback the remaining vector bundles and sections on the already constructed real oriented blow-up.

\begin{prop}\label{prop : independant on the order}
	Up to homeomorphism, the resulting space \eqref{eq:iterated-blow-up} is independent on the order of the blow-ups.
\end{prop}

\begin{proof}
	Assume that we have vector bundles $E$ and $F$ and sections $\sigma\colon X\to E$ and $\tau\colon X\to F$. If we define $\blr_{\sigma,\tau}(X)$ by the following cartesian square
\[
\begin{tikzcd}
	\blr_{\sigma,\tau}(X)\rar\dar & \blr_{\sigma}(X)\dar{bl_\sigma}\\
	\blr_{\tau}(X)\rar{bl_\tau} & X,
\end{tikzcd}
\]
	then $\blr_{\sigma,\tau}(X)$ is clearly homeomorphic both to $\blr_{bl_\tau^*(\sigma)}\blr_{\tau}(X)$ and $\blr_{bl_\sigma^*(\tau)}\blr_{\sigma}(X)$, which we abusively denoted $\blr_\sigma\blr_\tau(X)$ and $\blr_\tau\blr_\sigma (X)$. This proves the case $n=2$, and the general case $n\geq2$ follows by induction. 
\end{proof}

\subsection{Blowing-up a submanifold}

Let $X$ be a smooth manifold and $Z\subset X$ be a codimension $c$ submanifold. If $Z$ is defined by the zero set of a smooth map $\sigma\colon X\to \R^c$, we define
\[\blr_Z(X)\coloneqq\blr_{\sigma}(X),\]
where we view $\sigma$ as a section of the trivial $\R^c$-bundle on $X$. In general, $Z$ is locally of this form and we define $\blr_Z(X)$ by glueing the local blowups. This is verified in detail in \cite{AK} (note that they define blow-ups slightly differently, but equivalently thanks to Example~\ref{exam:AK}).

\subsection{Blowing-up a subscheme}

There is also the algebro-geometric blow-up $\bl_Z(X)$ when $Z\subset X$ is a closed subscheme, see any book on schemes. This is related to the topological construction thanks to the following.

\begin{prop}
	Assume $X$ is a smooth complex scheme and $Z\subset X$ a smooth closed subscheme of codimension $c$ defined as the zero locus of a section $\sigma\colon \mathcal{O}_X\to\mathcal{E}$ with $\mathcal{E}$ a locally free sheaf of rank $c$. Then there is a homeomorphism
\[(\bl_ZX)(\C)\cong \blc_{\sigma}(X(\C)).\] 
\end{prop}

We shall need the following important proposition.

\begin{prop}\label{prop : real blow-up is real-complex blow-up}
    Let $X$ be a smooth complex scheme and $Z\to X$ be a smooth codimension $c$ subscheme. Then there is a homeomorphism over $X(\C)$:
    \[\blr_{Z(\C)}{X(\C)}\cong\blr_{\wt{Z}(\C)}[(\bl_ZX)(\C)],\]
    where $\wt{Z}\subset\bl_Z(X)$ denotes the dominant transform of $Z$, which in this case is just the exceptional divisor.
\end{prop}

\begin{proof}
	If $Z$ is given by the vanishing locus of a map $\sigma\colon X\to\mathbf{A}^c$, then this  is  a straightforward computation. In general, $X$ will be covered by affine subschemes on which this is the case.
\end{proof}

\section{Log schemes}
\subsection{Log schemes and their blow-ups}
\begin{defn}
	A \emph{log structure} on a scheme $X$ is a finite collection $\{\sigma_i\colon\mathcal{O}_X\to\mathcal{L}_i\}_{i\in I}$ of line bundles on $X$ each equipped with a section. The pair $\logX=(X,\{\sigma_i\})$ is called a \emph{log scheme}. Given two log structures $\{\sigma_i:\mathcal{O}_X\to\mathcal{L}_i\}$ and $\{\tau_j:\mathcal{O}_X\to \mathcal{M}_j\}$ on the same scheme $X$, a morphism from $\{\sigma_i\}$ to $\{\tau_j\}$ is the data of natural numbers $e_{ij}$ and isomorphisms
	\[\mathcal{L}_i\cong\bigotimes_{j}\mathcal{M}_j^{\otimes e_{ij}}\]
	identifying the section $\sigma_i$ with $\bigotimes_j\tau_j^{\otimes e_{ij}}$.
\end{defn}

\begin{defn}
	Given a morphism of schemes $f\colon Z\to X$ and a log structure $\{\sigma_i\}_{i\in I}$ on $X$, there is an \emph{induced log structure} on $Z$ given by $\{f^*(\sigma_i)\colon\mathcal{O}_Z\to f^*\mathcal{L}_i\}_{i\in I}$.

A morphism of log-schemes $\logX=(X,\{\sigma_i\})\to\logY=(Y,\{\tau_j\})$ is a map of schemes $f:X\to Y$ together with a morphism of log-structures  from the log-structure of $X$ to the log-structure pulled-back along $f$.
\end{defn}

\begin{example}\label{exam : main example of log scheme}
	For $Z\subset X$ a smooth codimension $1$ subscheme, we have a corresponding line bundle $\mathcal{O}_X(Z)$ with a canonical section $\mathcal{O}_X\to \mathcal{O}_X(Z)$. This gives a log-structure on $X$ that we denote by $\logX$. There is also a log structure on $Z$ given by the normal bundle of the inclusion $Z\to X$ and the zero section. This log-structure denoted $\logZ$ is simply the induced log structure along the inclusion $Z\subset X$. It follows that $\logZ\to \logX$ is a morphism of log schemes.
	
	More generally, if we have a finite collection $Z_i$ of codimension $1$ subschemes (for example, a normal crossing divisor), we have the corresponding log scheme $\logX=(X,\{\sigma_i\})$.
\end{example}

There is a natural blow-up construction in the category of log schemes.
\begin{defn}\label{def: log blow-up}
	Let $\logX=(X,\{\sigma_i\})$ be a smooth log scheme and $Z$ be a smooth closed subscheme of positive codimension. We define the \emph{logarithmic blow-up of $\logX$ along $Z$} as the log scheme
\[\bll_Z(\logX)=\big(\bl_Z(X),\, \{bl_Z^*\sigma_i\}\sqcup \{\sigma'\}\big),\]
	where $bl_Z\colon\bl_Z(X)\to X$ is the blow-down map and $\sigma'\colon\mathcal{O}_X\to\mathcal{O}_X(\wt{Z})$ is the line bundle corresponding to the divisor $\wt{Z}$. 
\end{defn}

\subsection{The Kato--Nakayama realization}

\begin{defn}\label{def:KN}
	For a log scheme $\logX=(X,\{\sigma_i\colon\mathcal{O}_X\to\mathcal{L}_i\}_{1\leq i\leq n})$ over $\C$, the \emph{Kato--Nakayama space} is the topological space defined as the iterated real blow-up of its complex points:
\[\kn(\logX)\coloneqq\blr_{\sigma_n}\blr_{\sigma_{n-1}}\ldots\blr_{\sigma_1}[X(\C)].\]
\end{defn}

\begin{rem}
	Note that $\mathcal{L}_i$ are complex line bundles on $X$, that are in the definition of $\kn(\logX)$ viewed as real plane bundles on $X(\C)$. The construction of $\kn(\logX)$ does not depend on the chosen order of blow-ups thanks to Proposition \ref{prop : independant on the order}.
\end{rem}

\begin{example}\label{exam : main example of log scheme - KN}
	As in Example~\ref{exam : main example of log scheme} let $Z_i$, $1\leq i\leq n$, be codimension $1$ subschemes, giving the log scheme $\logX=(X,\{\sigma_i\colon
    \mathcal{O}_X\to\mathcal{O}_X(Z_i)\}_{1\leq i\leq n})$. If $n=1$, then $\kn(\logX)$ is simply the oriented real blow-up of $Z_1$ and is a manifold with boundary. More generally, if $\bigcup_iZ_i$ forms a normal crossing divisor, then $\kn(\logX)$ is a smooth manifold with corners.
\end{example}

\begin{prop}\label{prop : kn of log blow-up}
	If $\logX=(X,\{\sigma_i\colon
 \mathcal{O}_X\to\mathcal{L}_i\}_{1\leq i\leq n})$ is a smooth log scheme and $Z$ is a smooth closed subscheme of positive codimension, then there is a homeomorphism
\[\kn(\bll_Z(\logX))\cong \blr_{\kn(\logZ)}(\kn(\logX)).\]
\end{prop}

\begin{proof}
	Working locally, we may assume that $X$ is a smooth affine scheme and $Z$ is a complete intersection, i.e.\ $Z=\sigma^{-1}(0)$ for $\sigma\colon X\to\mathbf{A}^c$ an algebraic map and $c=\on{codim}(Z)$.
	On one hand, by Definition~\ref{def: log blow-up} we have $\bll_Z(\logX)=\bl_Z(X)$ with the log structure $\{bl_Z^*\sigma_i\}\cup\{\sigma'\}$, so by Definition~\ref{def:KN} we have
\[\kn(\bll_Z(\logX))=\blr_{\sigma_n}\ldots\blr_{\sigma_1}\blr_{\sigma'}[(\bl_ZX)(\C))].\]
	By Proposition~\ref{prop : real blow-up is real-complex blow-up} we can omit the complex blow-up, so the last space is homeomorphic to 
\[\blr_{\sigma_n}\ldots\blr_{\sigma_1}\blr_{\sigma}[X(\C)].\]
	By Proposition~\ref{prop : independant on the order} we can change the order of the blow-ups, so this is homeomorphic to
\[\blr_{\sigma}\blr_{\sigma_n}\ldots\blr_{\sigma_1}[X(\C)].\]
	Since by definition $\kn(\logX)=\blr_{\sigma_n}\blr_{\sigma_{n-1}}\ldots\blr_{\sigma_1}[X(\C)]$, the last space is immediately identified with $\blr_{\kn(\logZ)}(\kn(\logX))$.
\end{proof}

We shall also need the following proposition.

\begin{prop}\label{prop : KN of a pullback}
    Let $\logX=(X,\{\sigma_i\})$ be a log scheme and $f\colon Z\to X$ a map of smooth schemes. Let $\logZ=(Z,\{f^*\sigma_i\})$ be the induced log scheme. Then there is a cartesian square of topological spaces
\[
\begin{tikzcd}
	\kn(\logZ)\rar\dar & \kn(\logX)\dar\\
	Z\rar{f} & X.
\end{tikzcd}
\]
\end{prop}

\begin{proof}
    This can be proved by induction on the number of line bundles in the log-structure. The base case and the induction step are proved using Proposition~\ref{prop : cartesian square}.
\end{proof}


\section{Compactifications of configuration spaces}

\subsection{Stratified spaces}

In this paper, it will be convenient to use the following definition of stratified spaces.

\begin{defn}
Let $P$ be a poset. A $P$-stratified space is a functor $X:P\to\cat{Top}$ such that for each map $p\leq q$ in $P$, the induced map
\[X(p)\to X(q)\]
is a closed inclusion.
\end{defn}

\begin{rem}
If $X$ is a stratified space in the usual sense, we can take $P$ to be the poset of strata and obtain a functor as above sending $p$ to the closure of the corresponding stratum. We will write $X(p)^{\circ}$ for the \emph{open} stratum indexed by $p$, so that
\[X(p)=\overline{X(p)^{\circ}}\]
and
\[X(p)^{\circ}=X(p)-\bigcup_{q<p}X(q)\] 
\end{rem}

If our poset $P$ has a maximal element $\infty$, then all the values of a $P$-stratified space $X$ are closed subspaces of $X(\infty)$. In this situation, we shall say that $X$ is a $P$-stratification of $X(\infty)$.

\begin{defn}\label{defi : pullback stratification}
Let $P$ be a poset with a maximal element $\infty$ and $X:P\to\cat{Top}$ be a $P$-stratified space. Let $f:Y\to X(\infty)$ be a continous map. We define the pullback stratification on $Y$ along $f$ to be the functor
$P\to\cat{Top}$
sending $p$ to $f^{-1}(X(p))\subset Y$.
\end{defn}

\subsection{The Axelrod--Singer compactification}
\label{subsec:AS}

Let $X$ be a \emph{smooth manifold without boundary} and $n$ a positive integer. The ordered configuration space $\Conf_n(X)$ of $n$ points in $X$ is the interior of a smooth manifold with corners $\AS_n(X)$ constructed in \cite[Section 5]{axelrodsinger} and called the Axelrod--Singer compactification. (Note, however, that despite the name $\AS_n(X)$ is only compact if $X$ itself is compact.) This construction has the following properties.
\begin{enumerate}
	\item $\AS_n(X)$ is a stratified space whose strata are indexed by $\forset(n)$, the poset of forests on $\{1,\ldots n\}$. We denote by $\AS_{\Phi}(X)$ the closure of the stratum indexed by a forest $\Phi$.
	\item The stratum corresponding to the unique forest with $n$ trees is the interior of $\AS_n(X)$ and is diffeomorphic $\Conf_n(X)$.
	\item The codimension of $\AS_{\Phi}(X)^\circ$ is the number of elements of $\Phi$ that have strictly more than one element.
\end{enumerate}

\subsection{The Fulton MacPherson compactification}
\label{subsec:FM}
	Let $X$ be a \emph{smooth algebraic variety} over a field $\K$ and $n$ a positive integer. By an algebraic variety, we mean an integral separated scheme over $\K$. The ordered configuration space $\Conf_n(X)$ of $n$ points in $X$ has an analogous compactification $\FM_n(X)$, a predecessor to the Axelrod-Singer compactification due to Fulton and MacPherson~\cite{fultonmacpherson}, that has the following properties.
\begin{enumerate}
	\item $\FM_n(X)$ is stratified by subvarieties $\FM^{\circ}_\Phi(X)$ indexed by elements $\Phi \in \for(n)$, the poset of forests on $\{1,\ldots, n\}$.
	\item The subvariety corresponding to the unique forest with $n$ trees is the complement of a normal crossing divisor and is isomorphic to $\Conf_n(X)$.
	\item The components of this divisor are indexed by forests with only one element of cardinality $\geq 2$.
	\item If $\Phi$ and $\Psi$ are two forests on $\{1,\ldots,n\}$ such that $\Phi\cup\Psi$ is still a forest, then
	\[\FM_{\Phi}(X)\cap\FM_{\Psi}(X)=\FM_{\Phi\cup\Psi}(X) \; .\]
\end{enumerate}

\subsection{Forgetting points}\label{subsection: forgetting points}

Let $i\colon S\to T$ be an injection between finite sets. Precomposition with this map induces a projection map $p\colon X^T\to X^S$ given by
\[(x_t)_{t\in T}\mapsto(x_{i(s)})_{s\in S}.\]
This map restricts to a map
\[\Conf_T(X)\to \Conf_S(X)\]
We claim that this map can be extended to the Axelrod--Singer compactification if $X$ is a smooth manifold and to the Fulton--MacPherson compactification if $X$ is a smooth algebraic variety. We explain this for the latter case, the former being similar. 

For a finite set $R$, the variety $\FM_R(X)$ is obtained by blowing up the diagonals $\Delta_U$ for $U\subset R$ with $|U|\geq 2$ in order of increasing dimension. Then we define the map
\[p\colon\FM_T(X)\to \FM_S(X)\]
as the composite of two maps
\[\FM_T(X)\to \FM_{T,S}(X)\to \FM_S(X)\]
where $\FM_{T,S}(X)$ is by definition the iterated blow-up of diagonals of $X^T$ indexed by subsets of $T$ of the form $i(U)$ for $U$ a subset of $S$ of cardinality at least $2$. The map $\FM_T(X)\to \FM_{T,S}(X)$ is then simply a blow-down map (we blow up more things in the source than in the target) while the second map $\FM_{T,S}(X)\to \FM_S(X)$ is induced by functoriality of blow-ups.

The map $p\colon\FM_T(X)\to\FM_S(X)$ is compatible with the stratification. Let us explain precisely what this means. The source is stratified by the poset $\for(T)\op$ while the target is stratified by $\for(S)\op$. The injective map $S\to T$ induces a map of posets $\for(T)\to\for(S)$ so we may view both schemes as stratified by $\for(T)$ and it is straightforward to verify that the map $p$ extend to a natural transformation of functors $\for(T)\op\to\cat{Var}_{\K}$.

\subsection{The configuration category}

\begin{prop}
    For a smooth manifold $X$ without boundary, there is a functor $\for\op\to\cat{Top}$ whose value on a forest $(S,\Phi)$ is the closed stratum $\AS_{\Phi}(X)\subset \AS_S(X)$.
\end{prop}

\begin{proof}
    First, we show that the assignment $(S,\Phi)\to\AS_{\Phi}(X)$ can be promoted to a functor $(\for')\op\to \cat{Top}$. By definition, $\for'$ is generated by maps of the form $\id_S\colon(S,\Phi)\to (S,\Psi)$ with $\Phi\subset \Psi$ and maps of the form $f\colon(S,\Phi)\to (T,f^{-1}\Phi)$ with $f\colon T\hra S$ an injection. To a map of the first kind we assign the inclusion of strata
        \[\AS_\Psi(X)\hra\AS_{\Phi}(X)\]
    and to a map of the second kind, we assign the map
        \[\AS_{\Phi}(X)\ra \AS_{f^{-1}\Phi}(X)\]
    induced by forgetting points and changing labels according to the map $f$. This defines a functor $(\for')\op\to\cat{Top}$. Moreover, the factorization to $\for\op$ is trivially checked.
\end{proof}

If we precompose the functor from the proposition with the functor
    \[\Delta\fin\to\for\]
given in section \ref{sec:forestswlevels}, we obtain a functor
    \[\con^{AS}(X)\colon\Delta\fin\op\to\cat{Top}.\]

\begin{thm}
    The functor $\con^{AS}(X)$ is the Axelrod--Singer model of the configuration category, in the sense of \cite{boavidaspaces}.
\end{thm}

\begin{rem}\label{rem:con}
Recall from \cite{boavidaspaces} that the configuration category of the manifold $M$ -- denoted $\con(M)$ -- is the category whose space of objects is the disjoint union of configuration spaces $\Conf_S(M)$, $S$ a finite set (possibly the empty set), and a morphism from a configuration $x : R \hookrightarrow M$ to a configuration $y : S \hookrightarrow M$ is a pair $(f, H)$ where $f : R \to S$ is a map of finite sets, and $H$ is a (reverse) exit path in the stratified space $M^R$, stratified by the set of equivalence relations on $R$. The configuration category has an obvious reference map to the category of finite sets, and it is a good idea to take the nerve and regard $\con(M)$ as a simplicial space over the nerve of $\fin$. Given a simplicial space over $\fin$, $p : X \to \fin$, we may produce a functor $\Delta \fin \to \cat{Top}$ whose value at $\alpha \in \fin_k$ is simply $p^{-1}(\alpha)$. Conversely, given a functor $F : \Delta \fin \to \cat{Top}$ we may produce a simplicial space $X$ over $\fin$ by declaring that a $k$-simplex of $X$ is a pair $(f, z)$ where $f \in \fin_k$ and $z \in F(f)$. This induces an equivalence of ($\infty$-)categories.
\end{rem}

\begin{proof}[Sketch proof]
This is stated in \cite[Section 3]{boavidaspaces}, where the Axelrod-Singer model is called the \emph{screen completion} model, but the sketch there is incomplete. A more detailed account appears in \cite{Joao-thesis}. We will summarize that here. To compare $\con(M)$ and $\con^{AS}(M)$ we use the remark above to first view $\con^{AS}(M)$ as a simplicial space (actually, topological category) over the nerve of $\fin$, which we'll denote by $C$ for the duration of this proof.  That is, a $k$-simplex of $C$ is a pair $(f, z)$ where $f : S_0 \to \dots \to S_k$ is a $k$-simplex in $\fin$ and $z \in \con^{AS}(M)(f)$. We will describe a zigzag of the form
\[
C \to \mathcal{A} \gets \mathcal{B} \to \con(M)
\]
where $\mathcal{A}$ and $\mathcal{B}$ are certain topological categories over $\fin$, to be defined.

The category $\mathcal{A}$ has the same objects as $C$. A morphism in $\mathcal{A}$ is a pair $(a, H)$, where $a$ is a morphism in $C$ and $H$ is a (Moore) piecewise smooth exit path in the space of objects of $C$ ending at the source of $a$. The source of the morphism $(a,H)$ is $H(0)$ and the target is the target of $a$. Composition is by parallel transport. The functor $C \to \mathcal{A}$ is the identity on objects and on morphisms it is the inclusion of constant paths.

To describe the second auxiliary category $\mathcal{B}$, we equip $M$ with a Riemannian metric. An object in the category $\mathcal{B}$ is a pair $(x, \rho_x)$ where $x : \{1, \dots k\} \hookrightarrow M$ is a configuration and $\rho_x \in (\R_{> 0})^k$ is a tuple of positive real numbers subject to the following conditions:
\begin{enumerate}
\item for each $i \in \{1, \dots k\}$, the ball $B_i$ of radius ${\rho_x}_i$ centered at $x_i$ is geodesically convex (i.e. any two points in it are connected by a unique geodesic in that ball) and ${\rho_x}_i$ is less than the injectivity radius at $x_i$.
\item $B_i \cap B_j = \varnothing$ if $i \neq j$.
\end{enumerate}
The (disjoint) union of the $B_i$ will be denoted $B_x^\rho$. A morphism in $\mathcal{B}$ from $(x, \rho_x)$ to $(y, \tau_y)$ exists if $B_x^\rho \subset B_y^\tau$ and is given by an exit path in $B_y^\tau$ from $x$ to $y$ which is \emph{piecewise geodesic}. By a piecewise geodesic exit path we mean an exit path $H$ on ${B_y^\tau}$ for which one can find a finite partition of the interval into subintervals so that the restriction of $H$ to each subinterval is a geodesic. The functor $\mathcal{B} \to \con(M)$ forgets the balls, i.e. sends a pair $(x, \rho_x)$ to $x$. 

We are left to describe the functor $\mathcal{B} \to \mathcal{A}$. On objects, it also sends a pair $(x, \rho_x)$ to $x$, viewed as a configuration in the interior of the compactified configuration space. A morphism $(x, \rho_x) \to (y, \rho_y)$ in $\mathcal{B}$ given by a piecewise geodesic exit path $H$ from $x$ to $y$, with $x$ a configuration of $k$ points in $M$, gives rise to an exit path in the Axelrod-Singer \emph{compactification} of the configuration space of $k$ points in $M$. This is the main point, and why we need piecewise geodesic paths. An arbitrary path might not give rise to a path in the compactification, e.g. if one point is orbiting around another, being pulled towards it. 

Having defined these three functors, showing that they are equivalences, in fact degreewise weak equivalences of Segal spaces, is more mechanical. On spaces of objects, this is immediate. Then one observes that, in the four categories, the target map from the space of morphisms to objects is a fibration. And verifies, for each functor, that the induced map between fibers of the target maps is a weak equivalence.
\end{proof}

Similarly, we can prove the following proposition. Let us denote by $\cat{Var}_{\K}$ the category of smooth algebraic varieties over $\K$.

\begin{prop}
    For a smooth algebraic variety $X$ over $\K$, there is a functor $\for\op\to\cat{Var}_{\K}$ whose value on a forest $(S,\Phi)$ is the closed stratum $\FM_{\Phi}(X)\subset \FM_S(X)$.
\end{prop}

Again we can precompose this functor with the functor
$\Delta\fin\to\for$
of section \ref{sec:forestswlevels} and obtain a functor
\[\con^{FM}(X)\colon\Delta\fin\op\to\cat{Sch}_{\K},\]
that we call the \emph{Fulton--MacPherson model of the configuration category}.

\begin{rem}
    If $\K$ is the field of complex numbers, this object does not have the correct homotopy type in the sense that its analytification is not the configuration category of the analytification of $X$. However, our main Theorem, Theorem~\ref{thm:main}, asserts that this object can be equipped with a log-structure whose Kato--Nakayama realization is exactly the object $\con^{AS}(X(\C))$ constructed in the previous paragraph.
\end{rem}


\section{Wonderful compactifications and the Kato--Nakayama construction}

Our main result is about configuration spaces and their compactification. However, with no additional effort we can prove a more general result about wonderful compactifications so we take this small detour.
 
 \subsection{Wonderful compactifications}
We follow the conventions of \cite{liwonderful}.

\begin{defn}\label{def:wonderful}
	Fix a smooth algebraic variety $Y$ over $\C$ and a finite collection $\G$ of smooth closed subvarieties of $Y$. Let $\mathcal{S}$ be the set of all possible intersections of some elements of $\G$. Then $\G$ is called a \emph{building set} if for every $S\in\mathcal{S}$ the minimal elements in $\{G\in\G:G\supseteq S\}$ intersect transversely (the tangent bundles span $TY$) and their intersection is $S$. 

    The \emph{wonderful compactification} $Y_{\G}$ associated to this data is defined as the closure of
        \[Y-\bigcup_{\G}G\subset\prod_{G\in\G}\bl_{G}Y.\]
\end{defn}

It was shown in \cite{liwonderful} that $Y_{\G}$ can be constructed as an iterated blow-up
    \begin{equation}\label{eq:Li}
    Y_\G=\bl_{\wt{G}_n}\bl_{\wt{G}_{n-1}}\ldots\bl_{G_1}Y,
    \end{equation}
where $\G=\{G_1,\ldots,G_n\}$ is an ordering such that $\{G_1,\ldots,G_i\}$ is a building set for each $i\leq n$. 

The variety $Y_\G$ contains  a normal crossing divisor whose components are indexed by the elements of $\G$. 
The complement of this normal crossing divisor is isomorphic to the complement of the arrangement of subvarieties generated by $\G$ (i.e.\ the closure of $\G$ under intersection). 
\begin{defn}
    We denote by $\logY_\G$ the log scheme corresponding to this normal crossing divisor in $Y_\G$ (see Example~\ref{exam : main example of log scheme}). 
\end{defn}

We record the following proposition whose proof is immediate.

\begin{prop}\label{prop : iterated log blow up}
The log-scheme $\logY_\G$ is the iterated logarithmic blow-up of $(Y,\G)$, that is,
\[\logY_\G=\bl^{log}_{\wt{G}_n}\bl^{log}_{\wt{G}_{n-1}}\ldots\bl^{log}_{G_1}Y \; .\]
\end{prop}

\begin{defn}
    In the real oriented world, we may perform the same sequence of blow-ups of  $Y(\C)$ instead. This gives a topological space that we denote by $Y_{\G}^\R$. 
\end{defn}

Let $(Y,\G)$ be as in Definition~\ref{def:wonderful}. Then the variety $Y$ is stratified by $P(\mathcal{G})$, the poset of subsets of $\G$, with the stratum indexed by $\mathcal{H}\subset\mathcal{G}$ given by $\bigcap_{G\in\mathcal{H}}G$. We make the convention that the closed stratum corresponding to the empty set is $Y$. These strata determine a functor
    \[P(\mathcal{G})\op\to \cat{Var}_{\K}\]
which, taking $\C$-points, induces a functor
    \[P(\mathcal{G})\op\to \cat{Top}.\]
Note that each time we blow up one element of $\G$, we obtain a variety stratified by the same poset (we replace each subvariety by their dominant transform). In particular, the wonderful compactifications $Y_{\G}$ and $Y_{\G}^\R$ are both stratified by the poset $P(\G)\op$.

\begin{rem}\label{rem : nests}
    For $Y_{\G}$ and $Y_{\G}^\R$, the poset $P(\G)$ can be replaced by the subposet $N(\G)$ of $\G$-nests \cite[Definition 2.3]{liwonderful} as these are the only elements of $P(\G)$ whose corresponding stratum is non-empty.
\end{rem}

\begin{example}\label{ex:wonderful-FM-AS}
	Our main examples are compactifications of configuration spaces $\Conf_n(X)$. 
 
    In the algebro-geometric case, $X$ is a smooth algebraic variety, and we take $Y=X^n$ and $\G$ the set of diagonals $\Delta_S$ for every $S\subset \{1,\ldots,n\}, |S|\geq 2$. Then the arrangement of subvarieties generated by $\G$ is the arrangement of polydiagonals, and the construction $Y_\G$ is isomorphic to the variety $\FM_n(X)$, the Fulton--MacPherson compactification of $\Conf_n(X)$ from Section~\ref{subsec:FM}, see \cite[Sec.4.2]{liwonderful}. 
 
    In the topological case, we instead consider $Y(\C)=X(\C)^n$, and the construction $Y_{\G}^\R$ is isomorphic to the Axelrod--Singer compactification $\AS_n(X(\C))$ from Section~\ref{subsec:AS}. 

	In both cases the poset $P(\G)$ is the poset of family of subsets of $\{1,\ldots,n\}$ with at least two elements. The poset $N(\G)$ is isomorphic to $\for(n)$ (see Remark \ref{rem : forests vs nests}).
\end{example}

\subsection{Kato--Nakayama commutes with wonderul compactifications}

Our main theorem in the general setting of wonderful compactifications is the following.

\begin{thm}\label{theo : KN of wonderful}
	There is a homeomorphism 
\[\kn(\logY_\G)\cong Y_\G^\R.\]
    Moreover, this homeomorphism identifies the stratification on $Y_\G^\R$ indexed by $P(\G)\op$ with the pullback (Definition \ref{defi : pullback stratification}) of the stratification of $Y_\G(\C)$ along the projection map
\[\kn(\logY_\G)\to Y_\G(\C).\]
\end{thm}

\begin{proof}
    We first prove the statement ignoring the stratification. The proof is by induction on the length of the ordered set $\G=\{G_1,\ldots, G_n\}$. For $n=1$, we use the definition $\logY_\G=(\bl_{G_1}Y,\wt{G}_1)$ and Proposition \ref{prop : real blow-up is real-complex blow-up}:
\[\kn(\logY_\G)\coloneqq\blr_{\wt{G}_1(\C)}(\bl_{G_1}Y(\C))\cong\blr_{G_1(\C)}(Y(\C))\eqqcolon Y_\G^\R.\]
    Assume that the proposition has been proved for length $n-1$, and denote $\mathcal{H}=\{G_1,\ldots,G_{n-1}\}$. By Proposition \ref{prop : iterated log blow up}, we have that
\[\logY_\G=\bll_{\wt{G}_n}(\logY_{\mathcal{H}}).\]
    Thus, Proposition~\ref{prop : kn of log blow-up} gives a homeomorphism
\[\kn(\logY_\G)
    \cong\blr_{\kn(\wt{G}_n)}(\kn(\logY_{\mathcal{H}}))
    \cong\blr_{\kn(\wt{G}_n)}(Y_{\mathcal{H}}^\R)\]
	where the second homeomorphism comes from the induction hypothesis. The right hand side of this equation is exactly $Y_\G^\R$. (Indeed, $\kn(\wt{G}_n)$ agrees with the inverse image of $\wt{G}_n$ under the projection from the real blowup to the complex blowup.)
		
	Now, we show that this homeomorphism is compatible with the stratification. We start with a stratum corresponding to one element $G$ of $\G$. In this case the corresponding closed stratum in $Y_\G$ is a divisor $D_G$ in $Y_\G$. By Proposition \ref{prop : KN of a pullback}, if we give $D_G$ the log-structure induced by the inclusion $D_G\to Y_\G$, we have a cartesian square
	\[
\begin{tikzcd}
	\kn(\mathsf{D}_G)\rar\dar & \kn(\logY_\G)\cong Y_\G^{\R}\dar{p}\\
	D_G\rar{} & Y_\G,
\end{tikzcd}
\]
    On the other hand, the closed stratum $D_G^\R$ indexed by $G$ in $Y_\G^{\R}$ is the dominant transform of $G$ in $Y_\G^\R$ which can easily be identified with $p^{-1}(D_G)$ if we track it through the sequence of blow-ups. This shows compatibility for the stratum for $G$.
	
	 Consider now the closed stratum $Z_S$ in $Y_\G$ indexed by $S\subset \G$. By definition, we have
\[Z_S=\bigcap_{G\in S}D_G.\]
    Then the corresponding closed stratum in the pullback stratification of $\kn(\logY_\G)$ is $\kn(\logZ_S)$ where $\logZ_S$ is given the log-structure pulled-back along the inclusion $Z_S\to Y_\G$. By Proposition \ref{prop : KN of a pullback}, we have a cartesian square
	\[
\begin{tikzcd}
	\kn(\logZ_S)\rar\dar & \kn(\logY_\G)\dar{p}\\
	Z_S\rar{} & Y_\G,
\end{tikzcd}
\]
    On the other hand, using the case $|S|=1$ that we have just treated, the stratum $Z_S^\R$ indexed by $S$ in $Y_\G^\R$ is given by
\[Z_S^\R=\bigcap_{G\in S} D_G^\R=\bigcap_{G\in S} p^{-1}(D_G)=p^{-1}(\bigcap_{G\in S}D_G)=p^{-1}(Z_S),\]
    which concludes the proof.
\end{proof}

\begin{cor}\label{cor:KN(FM)=AS}
    For any finite set $S$ and any smooth complex algebraic variety $X$, there is a homeomorphism of $\for(S)\op$-stratified spaces
\[\kn(\logFM_S(X))\cong \AS_S(X(\C)).\]
\end{cor}

\begin{proof}
    This is an immediate application of Theorem \ref{theo : KN of wonderful}, Remark \ref{rem : nests} and Example~\ref{ex:wonderful-FM-AS}.
\end{proof}

\section{The algebro-geometric configuration category}

\begin{prop}\label{prop : compatibility with forgetting points}
    Let $i\colon S\hra T$ be an injection. Then
\begin{enumerate}
    \item The forgetful map $p\colon\FM_T(X)\to\FM_S(X)$ can be promoted to a map of log schemes
\[p\colon\logFM_T(X)\to\logFM_S(X).\]
    \item Upon applying Kato--Nakayama realization it induces the forgetful map
\[p\colon\AS_T(X(\C))\to\AS_S(X(\C)).\]
\end{enumerate}
\end{prop}

\begin{proof}
(i) We use our factorization
\[\FM_T(X)\to \FM_{T,S}(X)\to \FM_S(X)\]
from subsection \ref{subsection: forgetting points}. We give $\FM_{T,S}(X)$ the log-structure obtained from the exceptional divisors corresponding to the diagonals that have been blown-up (see Example \ref{exam : main example of log scheme}) and denote by $\logFM_{T,S}(X)$ the resulting log scheme. Then the map
\[\FM_T(X)\to \FM_{T,S}(X)\]
is canonically a map of log schemes. Indeed, the pullback log structure on $\FM_T(X)$ is simply the log-structure corresponding to the diagonals indexed by the image of $i$.

On the other hand, the log-structure of $\FM_{T,S}(X)$ is also the pullback of the log-structure of $\FM_S(X)$ along the projection map $p\colon\FM_{T,S}(X)\to\FM_S(X)$. It follows that the map
\[\FM_{T,S}(X)\to \FM_S(X)\]
is also a map of log schemes.

(ii) We wish to identify two continuous maps $\AS_T(X(\C))\to\AS_S(X(\C))$. By definition they coincide on the interior of $\AS_T(X(\C))$ which is an open dense subset so they have to coincide on the whole space.
\end{proof}


Finally, we prove our main result.

\begin{thm}\label{thm:main}
    Let $X$ be a smooth algebraic variety. There exists a functor
\[\con^{log}(X)\colon\Delta\fin\op\to\cat{LogSch}_{\K}\]
    such that
\begin{itemize}
    \item composing with the forgetful functor $\cat{LogSch}_{\K}\to\cat{Sch}_{\K}$ recovers the functor $\con^{FM}(X)$.
    \item taking $\K=\C$ and applying Kato--Nakayama realization gives the functor $\con^{AS}(X(\C))$.
\end{itemize}
\end{thm}

\begin{proof}
We show that the functor $\for\op\to \cat{Sch}_{\K}$ given by
\[(S,\Phi)\mapsto \logFM_{\Phi}(X)\]
lifts to a functor to $\cat{LogSch}_{\K}$ and that the precomposition of the lifted functor with the map $\Delta\fin\op\to\for\op$ satisfies the properties in the statement. 

We have explained how $\logFM_{\Phi}(X)$ is a log scheme. 
Since maps in $\for$ are composites of maps of the form 
\[\id_S\colon(S,\Phi)\to (S,\Psi)\]
    with $\Phi \subset \Psi$ and maps of the form
\[i\colon(S,i^{-1}\Phi)\to (T,\Phi)\]
    with $i\colon S\to T$ an injection, it suffices to show that each of these maps induces map of log schemes. In the first case, this follows immediately from the definition of the log-structures on $\logFM_\Phi(X)$ and $\logFM_\Psi(X)$ (they are simply pullbacks of the log-structure of $\logFM_S(X)$). In the second case, this is Proposition~\ref{prop : compatibility with forgetting points}.

    Finally, this functor satisfies the above two requirements: this is obvious for the first and, for the second one, this is a consequence of Corollary~\ref{cor:KN(FM)=AS} and Proposition~\ref{prop : compatibility with forgetting points}.
\end{proof}

\subsection{The real wonderful compactification as a closure}

We conclude this section with a description of the real wonderful compactification as a closure. This is important as it reconciles the two existing definitions of the Axelrod--Singer compactification available in the literature: one through iterated blow-ups and the other as the closure of the inclusion of the configuration space in a product. Again, this is valid for any wonderful compactification. The case of configuration spaces can be found in \cite[Lemma C.1]{watanabeaddendum}.

\begin{prop}
	The closure of the image of the map
\[Y(\C)-\bigcup_{G\in\G}G(\C)\to \prod_{G\in\G}\blr_{G(\C)}(Y(\C))\]
	is isomorphic to $Y^\R_{\G}$.
\end{prop} 

\begin{proof}
	First, we observe that the map
\[Y(\C)-\bigcup_{G\in\G}G(\C)\to Y_\G^{\R}\]
	is injective with dense image. This follows from the fact that the target is a manifold with corners and the source is the interior of this manifold.

	Now, consider the following commutative diagram of topological spaces
\[\begin{tikzcd}
	Y(\C)-\bigcup_{G\in\G}G(\C)\ar[r]&Y_\G^\R\ar[r]\ar[d]&\prod_{G\in\G}\blr_{G(\C)}(Y(\C))\ar[d]\\
 	&Y_{\G}(\C)\ar[r]&\prod_{G\in\G}\blc_{G(\C)}(Y(\C))
\end{tikzcd}
\]
	According to \cite[8.4.3]{bergstromhyperelliptic}, the square is cartesian (indeed, we can view the product $\prod_{G\in\G}\bl_{G}(Y)$ as a log scheme in which we give each factor the log structure of Example \ref{exam : main example of log scheme} and then this square is exactly the square of \cite[8.4.3]{bergstromhyperelliptic}). The bottom horizontal map is a closed inclusion by \cite[Proposition 2.13]{liwonderful}, it follows that the top horizontal map is also a closed inclusion since the square is cartesian. This concludes the proof.
\end{proof}


\section{Galois action}

Let $\ell$ be a prime number, let $\Q_\ell$ be the field of $\ell$-adic number. Let $q$ be a power of a prime number $p$ different from $\ell$. Let $\K$ be a characteristic zero field equipped with an embedding $\K\to\C$. In this case, we may consider the functor
\[X\mapsto Et(X\times_\K\overline{\K})^{\wedge}\]
from the category of schemes over $\K$ to the category of profinite homotopy types. Here $Et$ denotes the \'etale homotopy type functor (also known as the ``shape'' of the \'etale site of $X$) and $U\mapsto U^{\wedge}$ is the profinite completion functor. By functoriality, the object $Et(X\times_\K\overline{\K})^{\wedge}$ comes equipped with an action of the group $\mathrm{Gal}(\overline{\K}/\K)$. This induces a Galois action on the cochain complex $C^*(Et(X\times_\K\overline{\K})^{\wedge},\Z/\ell^n)$ and thus also on $C^*(Et(X\times_\K\overline{\K})^{\wedge},\Z_\ell)$ using the weak equivalence
\[C^*(Et(X\times_\K\overline{\K})^{\wedge},\Z_\ell)\simeq\mathrm{holim}_n C^*(Et(X\times_\K\overline{\K})^{\wedge},\Z/\ell^n).\]

We can lift the Frobenius along the surjective map
\[\rm{Gal}(\overline{\K}/\K)\twoheadrightarrow \rm{Gal}(\overline{\k}/\k)\]
and we obtain an automorphism of $\varphi$ of $Et(X\times_\K\overline{\K})^{\wedge}$. Moreover, Artin's comparison theorem together with the smooth base change theorem (see \cite[XVI, Theorem 4.1 and XII, Corollary 5.4]{sga4}) gives us a weak equivalence of profinite spaces
\[Et(X\times_\K\overline{\K})^{\wedge}\simeq X(\C)^{\wedge}\]

By work of Carchedi-Scherotzke-Sibilla-Talpo, we can extend this to log-schemes and we obtain the following proposition.

\begin{prop}
Let $\K\subset \C$ be a $p$-adic field with residue field $\k$ and $X$ be a smooth scheme over $\K$. Then the profinite completion of $\con^{AS}(X(\C))$ has an action of the absolute Galois group $\mathrm{Gal}(\overline{\K}/\K)$. In particular, it is equipped with an automorphism $\varphi$ lifting the Frobenius in $\mathrm{Gal}(\overline{\k}/\k)$.
\end{prop}

\begin{proof}
	There is an action on the profinite étale homotopy type of the infinite root stack of $\con^{log}(X)$. By \cite{CSST2017} this is homotopy equivalent to the profinite completion of the Kato--Nakayama space which is identified with the profinite completion of $\con^{AS}(X(\C))$ by Theorem \ref{thm:main}.
\end{proof}

We can extend this action to configuration categories of manifolds that are not necessarily complex algebrac varieties using the additivity result of the first author and Michael Weiss \cite{boavidaweissproduct}.

\begin{cons}\label{cons : galois action on product conf}
Let $\K\subset \C$ be a $p$-adic field and $X$ be a smooth scheme over $\K$ and $Y$ be a smooth manifold. Then we can consider the derived ``box product'' of the two configuration categories which is weakly equivalent to the configuration category of the product of manifolds:
\[\con^{AS}(X(\C)\times Y)\simeq \con^{AS}(X(\C))\boxtimes^L\con^{AS}(Y)\] 
by the main theorem of \cite{boavidaweissproduct}. We refer the reader to \cite{boavidaweissproduct} for details about the box product construction. Moreover by \cite[Theorem 5.5]{boavidaformality}, there is a weak equivalence
\[(\con^{AS}(X(\C))\boxtimes^L\con^{AS}(Y))^{\wedge}\simeq \con^{AS}(X(\C))^{\wedge}\boxtimes^L\con^{AS}(Y)^{\wedge}.\]
Using the Galois action on the first factor, we equip the diagram of $\ell$-complete spaces $(\con^{AS}(X(\C)\times Y))^{\wedge}$ with a Galois action.
\end{cons}

\section{A formality result}

Our goal now is to use the Galois action on the $\ell$-completion of $\con^{AS}(X(\C)\times\R)$ given by Construction \ref{cons : galois action on product conf} in order to prove a formality result.

\begin{defn}
We say that an element $\alpha$ of $\overline{\Q_\ell}$ is a \emph{Weil number of weight $n$} (relative to $q$) if for any choice of embedding $\sigma\colon\overline{\Q_\ell}\to\C$, we have
\[|\sigma(\alpha)|=q^{n/2}\]
\end{defn}

In the following we use the phrase ``Galois representation'' for a vector space over $\Q_\ell$ equipped with an automorphism $\varphi$. In practice, this automorphism will always come from the generator of $\mathrm{Gal}(\overline{\k}/\k)$ which explains the terminology.

\begin{defn}
Let $V$ be a $\Q_\ell$-vector space equipped with an automorphism $\varphi$. We say that $V$ is a \emph{pure Galois representation of weight $n$} if all the eigenvalues of $\varphi$ are Weil numbers of weight $n$.
\end{defn}

\begin{example}\label{exam : Tate twist}
For $n\in\mathbb{Z}$, we denote by $\Q_\ell(n)$ the Galois representation whose underlying vector space is the one-dimensional vector space $\Q_\ell$ and with $\varphi$ acting by multiplication by $q^n$. This is a Galois representation that is pure of weight $2n$.  More generally, for $V$ a $\Q_\ell$ vector space, we write $V(n)$ for the Galois representation with $V$ as an underlying vector space and $\varphi$ acting by multiplication by $q^n$.
\end{example}

The following is immediate.

\begin{prop}\label{prop : stability properties}

\begin{enumerate}
\item Let $0\to V'\to V\to V''\to 0$ be a short exact sequence of Galois representations. Then $V$ is pure of weight $n$ if and only if $V'$ and $V''$ are.
\item If $V$ and $V'$ are two Galois representations that are pure of weight $n$ and $m$ respectively, then $V\otimes V'$ is pure of weight $n+m$.
\item If $V$ is a Galois representation that admits a finite filtration by subrepresentations such that any term in the associated graded is pure of weight $n$, then $V$ is pure of weight $n$.
\end{enumerate}
\end{prop}

Now, let $X$ be a complex smooth algebraic variety of complex dimension $d$. We assume that $X$ comes from a smooth algebraic variety over a $p$-adic field $\K\subset\C$ so that, as explained above, the profinite completion of $X(\C)$ comes equipped with an action of a Frobenius lift. In order to simplify notation, we shall now simply write $X$ for the differentiable manifold $X(\C)$. We assume further that the following is true.

\begin{assum}\label{assumption}
For each $i\geq 0$, the group $H^i(X;\Q_\ell)$ is a representation that is pure of weight $i$.
\end{assum}

\begin{rem}
This assumption is satisfied if $X$ is smooth and proper and admits a smooth and proper lift $\mathcal{X}$ to $\mathcal{O}_K$ the ring of integers of $K$. Indeed, in that case, the Frobenius action on $H^*(X\times_\K\overline{\K},\Q_\ell)$ coincides with the Frobenius action on $H^*(\mathcal{X}\times_{\mathcal{O}_\K}\overline{\k},\Q_\ell)$ which is of the required form by Deligne's proof of the Weil conjecture \cite{deligneweil}. In practice, if one starts with a smooth and proper scheme over the complex numbers, then it is in fact defined over a finitely generated ring and by usual ``spreading-out''argument, we may assume that it the base change of a smooth and proper scheme $\mathcal{X}$ defined over $\mathcal{O}_\K$ for $\K$ a $p$-adic field with $p\neq \ell$.
\end{rem}

\begin{prop}\label{prop : relative term}
Under assumption \ref{assumption}, the relative cohomology group $H^{k}(X^2,\Conf_2(X);\Q_\ell)$ is a pure representation of weight $k$ for each $k\geq 0$.
\end{prop}

\begin{proof}
By the Thom isomorphism in \'etale cohomology, we have a Galois equivariant isomorphism
\[H^{k}(X^2,\Conf_2(X);\Q_\ell)\cong H^{k-2d}(X;\Q_\ell)\otimes\Q_{\ell}(d)\]
where $\Q_\ell(d)$ is defined in Example \ref{exam : Tate twist}
\end{proof}

\begin{thm}
Under assumption \ref{assumption}, the Galois representation $H^i(\Conf_n(X\times\R);\Q_\ell)$ is pure of weight $i$ for any value of $i$.
\end{thm}

\begin{proof}
In order to simplify notations, we drop the $\Q_\ell$ coefficients from the notation. The manifold $X\times\R$ has vanishing diagonal class. We thus have a presentation 
\[H^*(\Conf_n(X\times\R))\cong H^*(X^n)[x_{i,j}]/(relations)\]
where the classes $x_{i,j}$ are of degree $2d$ and are indexed by pairs $1\leq i<j\leq n$. This presentation can be found in \cite{BHK-homology}. The relations shall not concern us here. From this presentation and Proposition \ref{prop : stability properties}, we see that it is enough to prove the claim up to degree $2d$ since everything in higher degree is a product of classes of degree $\leq 2d$. By our assumption on $X$ and the K\"unneth isomorphism, the theorem is true up to degree $2d-1$. It thus remains to understand the Frobenius action on the classes $x_{i,j}$. Since the map
\[pr_{ij}:\Conf_n(X\times\R)\to\Conf_2(X\times\R)\]
is part of the configuration category structure, it becomes Galois equivariant after $\ell$-completion. The map $pr_{ij}^*$ sends $x_{12}\in H^{2d}(\Conf_2(X\times\R))$ to $x_{ij}$, therefore, we see that it is enough to prove that $H^{2d}(\Conf_2(X\times\R))$ is pure of weight $2d$. For this, using the fact that $X\times\R$ has vanishing diagonal class, we may use the short exact sequence (see \cite{BHK-homology})
\[0\to  H^{2d}(X^2)\to H^{2d}(\Conf_2(X\times\R))\to H^{2d+1}(X^2,\Conf_2(X\times \R))\to 0\]
We can then use \cite[Proposition 7.5]{boavidaformality} which claims that there is a Galois equivariant homotopy pushout square
\[
\begin{tikzcd}
\Conf_2(X)\times\Conf_2(\R)\ar[d]\ar[r]&\Conf_2(X)\times \R^2\ar[d]\\
X^2\times\Conf_2(\R)\ar[r]&\Conf_2(X\times\R)
\end{tikzcd}
\]
Let us denote by $V(i)$ the relative cohomology group $H^i(X^2,\Conf_2(X))$. Thanks to Proposition \ref{prop : relative term}, this representation is pure of weight $i$. By the above pushout square, the group $V(i)^{\oplus 2}$ is the relative cohomology group $H^i(\Conf_2(X\times\R),\Conf_2(X)\times\R^2)$. We thus have an exact sequence
\[V(2d)^{\oplus 2}\to H^{2d}(\Conf_2(X\times\R))\to H^{2d}(\Conf_2(X))\xrightarrow{\alpha} V(2d+1)^{\oplus 2}\]
Moreover, observe from the above pushout square that the map
\[\alpha:H^{2d}(\Conf_2(X))\to V(2d+1)^{\oplus 2}\]
factors as
\[H^{2d}(\Conf_2(X))\to H^{2d}(\Conf_2(X))^{\oplus 2}\to V(2d+1)^{\oplus 2}\]
where the first map is simply the diagonal map and the second map is the direct sum of two copies of the connecting map
\[H^{2d}(\Conf_2(X))\to H^{2d+1}(X^2,\Conf_2(X))\]
The kernel of this connecting map is pure of weight $2d$ since this is isomorphic to $H^{2d}(X^2)$. Therefore the kernel of $\alpha$ is pure of weight $2d$ and we deduce from the exact sequence above that $H^{2d}(\Conf_2(X\times\R))$ sits in the middle of a short exact sequence in which both other terms are pure of weight $2d$. By Proposition \ref{prop : stability properties} $H^{2d}(\Conf_2(X\times\R))$ is pure of weight $2d$ as desired. 
\end{proof}

\begin{cor}\label{cor}
The algebra $C^*(\Conf_n(X\times\R);\Q_\ell)$
is $\Sigma_n$-equivariantly formal as an $E_\infty$-algebra over $\Q_\ell$. In particular, $C^*(\Conf_n(X\times\R); \Q)$ is a formal cdga over $\Q$.
\end{cor}

\begin{proof}

This follows from the previous computation together with \cite{boavidahorelembindisks}. 
\end{proof}

We also get the following consequence.

\begin{prop}
Let $X$ be a smooth and proper algebraic variety over the complex numbers and assume that $H^*(X,\Q)$ is a Koszul algebra. Then $\Conf_n(X\times\R)$ is a Koszul space in the sense of Berglund (see \cite{berglundkoszul}).
\end{prop}

\begin{proof}
By the previous corollary this space is formal, so it suffices to show that $H^*(\Conf_n(X\times\R))$ is a Koszul algebra. This follows from \cite[Corollary 2.3]{bezrukavnikovkoszul}.
\end{proof}

As a consequence, we obtain a presentation for the Lie algebra of rational homotopy groups using \cite[Theorem 3]{berglundkoszul}. Examples of algebraic varieties $X$ satisfying our assumption include all smooth projective curves and all simply connected smooth projective algebraic surfaces.


\bibliographystyle{alpha}
\bibliography{conf}

\vspace{10pt}
\hrule

\end{document}